\pdfoutput=1
\documentclass[10pt]{amsart}

\usepackage{hyperref}
\hypersetup{colorlinks=false}
\title{On Adically Complete D-Modules in Characteristic Zero}
\author{Amnon Yekutieli}
\date{13 July 2024}
\address{Department of  Mathematics,
Ben Gurion University, Be'er Sheva 84105, Israel.}
\email{\href{mailto:amyekut@gmail.com}{amyekut@gmail.com}}
\urladdr{\url{https://sites.google.com/view/amyekut-math/home}}

\usepackage{libertine}
\usepackage[T1]{fontenc}
\usepackage[defaultsans, scaled=0.95]{droidsans}
\usepackage[libertine, varbb, libaltvw, liby, smallerops]{newtxmath}
\usepackage[cal=cm, calscaled=0.95, frak=euler, frakscaled=0.98]{mathalfa}
\usepackage[scaled=0.95]{inconsolata}

\usepackage{tikz}
\usepackage{tikz-cd}
\usetikzlibrary{arrows}
\tikzcdset{arrow style=tikz,
diagrams={>={Stealth[round,length=4pt,width=6pt,inset=3pt]}}}
\tikzcdset{every label/.append style = {font = \footnotesize}}
\usepackage{graphicx}

\newtheorem{thm}[equation]{Theorem}
\newtheorem{cor}[equation]{Corollary}
\newtheorem{prop}[equation]{Proposition}
\newtheorem{lem}[equation]{Lemma}
\theoremstyle{definition}
\newtheorem{dfn}[equation]{Definition}
\newtheorem{rem}[equation]{Remark}

\newtheorem{setup}[equation]{Setup}

\numberwithin{equation}{section}
\newcommand{\iso}{\xrightarrow{%
\smash{\raisebox{-0.5ex}{\ensuremath{\scriptstyle \simeq  \mspace{2mu}}}}}}
\newcommand{\inj}{\rightarrowtail} 
\newcommand{\surj}{\twoheadrightarrow}

\newcommand{\sub}{\subseteq}
\newcommand{\opn}{\operatorname}
\newcommand{\cat}[1]{\operatorname{\mathsf{#1}}}
\newcommand{\cd}{\mspace{0.8mu}{\cdotB}\mspace{0.8mu}}

\newcommand{\mfrak}[1]{\mathfrak{#1}}
\newcommand{\mcal}[1]{\mathcal{#1}}

\newcommand{\mrm}[1]{\mathrm{#1}}

\newcommand{\OO}{\mcal{O}}
\newcommand{\MM}{\mcal{M}}

\newcommand{\PP}{\mcal{P}}

\newcommand{\DD}{\mcal{D}}

\newcommand{\La}{\Lambda}
\newcommand{\si}{\sigma}
\newcommand{\la}{\lambda}

\renewcommand{\th}{\theta}

\newcommand{\al}{\alpha}

\newcommand{\ga}{\gamma}
\newcommand{\ep}{\epsilon}

\newcommand{\Om}{\Omega}

\renewcommand{\a}{\mfrak{a}}
\renewcommand{\b}{\mfrak{b}}

\newcommand{\bs}[1]{\boldsymbol{#1}}

\newcommand{\K}{\mathbb{K}}

\newcommand{\Q}{\mathbb{Q}}
\newcommand{\Z}{\mathbb{Z}}
\newcommand{\N}{\mathbb{N}}
\newcommand{\C}{\mathbb{C}}

\newcommand{\smfrac}[2]{\scalebox{1.1}%
{$\genfrac{}{}{0.45pt}{1}{#1}{#2 {}^{\vphantom{X}}}$}}
\newcommand{\tup}[1]{\textup{#1}}

\newcommand{\boplus}{\bigoplus\nolimits}
\newcommand{\ot}{\otimes}

\newcommand{\what}[1]{\widehat{#1}}  
\newcommand{\wh}[1]{\widehat{#1}}
\newcommand{\lto}{\leftarrow}

\renewcommand{\d}{\mathrm{d}}
\newcommand{\pa}{\partial}

\newcommand{\lb}{\linebreak}
\newcommand{\abs}[1]{\lvert #1 \rvert}
\newcommand{\sbmat}[1]{\left[ \begin{smallmatrix} #1%
\end{smallmatrix} \right]}
\newcommand{\lsp}{\mspace{1.0mu}}
\newcommand{\msp}{\mspace{2.5mu}}

\begin{document}

\begin{abstract}
Let $(X, \OO_X)$ be an algebraic manifold in characteristic $0$, or an
analytic manifold over $\C$.
A standard theorem says that a left $\DD_X$-module $\MM$, which is coherent as
an $\OO_X$-module, is locally free.
This theorem has a generalization to the adically complete algebraic setting,
in a paper by Ogus from 1973.

In the present paper we take a new look at the work of Ogus. We provide a
detailed proof of the theorem on $\DD$-modules, and extend it to the
non-noetherian setting. We also give another proof of an interesting
result of Ogus about adically complete modules
(slightly extended).

In the Appendix we discuss a related error in a book by Bj\"{o}rk.
\end{abstract}

\maketitle

\tableofcontents

\setcounter{section}{-1}
\section{Introduction}

Let $(X, \OO_X)$ be an algebraic manifold over a field $\K$ of characteristic
$0$, or an analytic manifold over $\C$. A standard theorem says that a left
$\DD_X$-module $\MM$, which is coherent as an $\OO_X$-module, is a locally free
$\OO_X$-module. See \cite{Be}, \cite{Bj} or \cite{HTT}.

This geometric theorem has a vast generalization to the adically complete
algebraic setting, namely Theorem 1.3 in the paper \cite{Og} by A. Ogus
from 1973. In Remark \ref{rem:305} we explain how the algebraic theorem of Ogus
implies the geometric theorem.

In our present paper we give a detailed proof of the theorem of Ogus,
and extend it to the non-noetherian setting. This is our Theorem \ref{thm:411}.
We also provide another proof of an interesting result of
Ogus about adically complete modules (see \cite[Proposition A1]{Og}),
and extend it slightly; this is our Theorem \ref{thm:410} below.
Thus our paper is not original research, but rather an attempt to throw new
light on these remarkable results.

Before stating our first theorem, we need a bit of background.
Consider a commutative ring $A$ and a finitely generated ideal $\a \sub A$.
Let $M$ be an $A$-module. We do not assume that $A$ is noetherian, nor that $M$
is finitely generated.
The module $M$ is equipped with its $\a$-adic
filtration $\{ \a^i \cd M \}_{i \in \N}$.
The {\em $\a$-adic completion} of $M$ is the $A$-module
$\La_{\a}(M) := \lim_{\lto i} \msp M / \a^i \cd M$, and $M$ is called
{\em $\a$-adically complete} (and separated) if the canonical homomorphism
$\tau_{M, \a} : M \to \La_{\a}(M)$ is an isomorphism.
If $\tau_{M, \a}$ is only injective, then $M$ is called {\em $\a$-adically
separated}.

Suppose the $A$-module $M$ has some other descending filtration
$\{ F^i(M) \}_{i \in I}$ by $A$-sub\-modules, indexed by a directed set
$(I, \leq)$. Namely, $\{ F^i(M) \}_{i \in I}$ is an inverse system of
submodules of $M$. The {\em $F$-adic completion} of $M$ is the $A$-module
$\La_{F}(M) := \lim_{\lto i} \msp M / F^i(M)$, and $M$ is called
{\em $F$-adically complete} (and separated) if the canonical homomorphism
$\tau_{M, F} : M \to \La_{F}(M)$ is an isomorphism.

Both the $\a$-adic filtration and the $F$-adic filtration determine $A$-linear
topologies on the $A$-module $M$. Let's denote the corresponding topological
spaces by $(M, \a)$ and $(M, F)$. We say that the $F$-adic topology
on $M$ is weaker than the $\a$-adic topology on it if the identity of $M$ is a
continuous function $(M, \a) \to (M, F)$. Concretely, this means that for every
$i \in I$ there exists $j \in \N$ such that $\a^j \cd M \sub F^i(M)$.

The first theorem is a slight improvement of \cite[Proposition A1]{Og}.

\begin{thm} \label{thm:410}
Let $A$ be a commutative ring and $\a \sub A$ a finitely generated ideal.
Let $M$ be an $A$-module, with a filtration
$\{ F^i(M) \}_{i \in I}$ indexed by a directed set $(I, \leq)$.
Assume that $M$ is $F$-adically complete, and the $F$-adic topology on $M$ is
weaker than the $\a$-adic topology. Then $M$ is $\a$-adically complete.
\end{thm}

This is repeated as Theorem \ref{thm:416} in Section \ref{sec:ad-comp} and
proved there. Note that in \cite{Og} it is assumed that the indexing set $I$ is
countable; but our proof does not require this.

Here are a few useful corollaries.

\begin{cor} \label{cor:410}
With $\a \sub A$ as in the theorem, suppose $\b \sub A$ is an ideal such that
$\a \sub \b$. If $M$ is $\b$-adically complete, then it is $\a$-adically
complete.
\end{cor}

The next is \cite[Corollary A3]{Og}.

\begin{cor} \label{cor:411}
With $\a \sub A$ as in the theorem, assume $M$ is $\a$-adically complete, and
$N \sub M$ is an $A$-submodule that is closed for the $\a$-adic topology. Then
$N$ is $\a$-adically complete.
\end{cor}

The next is a slight improvement of \cite[Corollary A2]{Og}.
Again we do not require the set $I$ to be countable.

\begin{cor} \label{cor:412}
With $\a \sub A$ as in the theorem, let
$\{ M_i \}_{i \in I}$ be an inverse system of $A$-modules,
indexed by a directed set $(I, \leq)$.
Assume every $M_i$ is an $\a$-adically complete module. Define the $A$-module
$M := \lim_{\lto i} M_i$. Then $M$ is $\a$-adically complete.
\end{cor}

The corollaries are repeated and proved in Section \ref{sec:ad-comp} of the
paper.

We now move to $\DD$-modules. Let $A$ be a commutative ring of characteristic
$0$, i.e.\ $A$ contains the field of rational numbers $\Q$.
We do not assume $A$ is noetherian.
Let $B := A[[t_1, \ldots, t_n]]$, the ring of power series over $A$ in $n$
variables, and let $\b \sub B$ be the ideal generated be the variables. We
know that $B$ is $\b$-adically complete.

Consider the ring $\DD_{B / A}$ of $A$-linear differential operators on $B$ (in
the sense of Gro\-then\-dieck, see \cite[Section 16]{EGA-IV}).
For every $j$ let $\pa_j := \pa / \pa t_j$, the partial derivative along $t_j$.
In Proposition \ref{prop:425} we show that $\DD_{B / A}$ is a Weyl algebra,
namely as a left $B$-module it is free with basis the collection of monomial
operators $\{ \pa_1^{k_1} \cdots \pa_n^{k_n} \}$, indexed by
$(k_1, \ldots, k_n) \in \N^n$, and the usual relations hold.

Let $M$ be a $B$-module. A left $\DD_{B / A}$-module structure on $M$ is the
same as an {\em integrable continuous connection}
$\nabla : M \to \wh{\Om}^1_{B / A} \ot_B M$.
Here $\wh{\Om}^1_{B / A}$ is the complete module of differential $1$-forms,
which is a free $B$-module with basis $(\d t_1, \ldots, \d t_n)$.
The passage between connections and $\DD$-modules is this: for an element
$m \in M$ there is equality
$\nabla(m) = \sum_j \msp \d t_j \ot \pa_j(m)$
in $\wh{\Om}^1_{B / A} \ot_B M$.
Unfortunately, we could not find a reference in the literature for this
interplay between $\DD$-modules and connections.
Anyhow, this is not required for our work, and it is mentioned only for the sake
of comparison with the work of Ogus.

Given a left $\DD_{B / A}$-module $M$, its {\em horizontal submodule} is the
$A$-module
\begin{equation} \label{eqn:410}
M^{\mrm{hor}} :=
\bigl\{ \msp m \in M \mid \pa_j(m) = 0
\ \tup{for all} \ j \msp \bigr\}  .
\end{equation}
In terms of the corresponding connection $\nabla$ we have
$M^{\mrm{hor}} = \opn{Ker}(\nabla)$.

The following theorem is a generalization of \cite[Theorem 1.3]{Og},
stated in terms of $\DD$-modules. Ogus works with integrable connections,
and assumes that the ring $A$ is noetherian. We do not put any finiteness
condition on $A$.

\begin{thm} \label{thm:411}
Let $A$ be a commutative ring containing $\Q$, let
$B := A[[t_1, \ldots, t_n]]$, and let
$\b \sub B$ be the ideal generated by the variables.
Suppose $M$ is a left $\DD_{B / A}$-module, which is $\b$-adically complete as
a $B$-module. Then the canonical $\DD_{B / A}$-module homomorphism
\[ B \ \wh{\ot}_A \ M^{\mrm{hor}} \to M \]
is bijective.
\end{thm}

In the theorem, $B \ \wh{\ot}_A \ M^{\mrm{hor}}$ is the $\b$-adic
completion of the $B$-module $B \ot_A M^{\mrm{hor}}$. It has an induced left
$\DD_{B / A}$-module structure. Ogus calls the corresponding connection a {\em
completely constant connection}.

The $1$-dimensional case of Theorem \ref{thm:411} is Theorem \ref{thm:240}.
For $n \geq 2$ it is Theorem \ref{thm:215}, which is proved using Theorem
\ref{thm:240} and induction on $n$. The induction requires Corollaries
\ref{cor:410} and \ref{cor:411}.

\begin{cor} \label{cor:450}
Under the assumptions of the theorem, the canonical $A$-module homomorphism
$M^{\mrm{hor}} \to M / \b \cd M$ is bijective.
\end{cor}

A $B$-module $M$ is called {\em $\b$-adically free} if $M$ is the $\b$-adic
completion of a free $B$-module.

\begin{cor} \label{cor:413}
Under the assumptions of the theorem, if $M / \b \cd M$ is a free $A$-module,
then $M$ is a $\b$-adically free $B$ module.
\end{cor}

The following corollary is a special case of the previous one.
It is \cite[Proposition 8.9]{Ka} (which is phrased in terms of connections).

\begin{cor} \label{cor:414}
Assume $A$ is a field and $M$ is a finitely generated $B$-module.
Then $M \cong B^{\oplus r}$ as $\DD_{B / A}$-modules, for some natural number
$r$.
\end{cor}

These corollaries are proved in Section \ref{sec:high-dim}.

Finally, in Appendix A, we discuss an error we discovered in the book
\cite{Bj}, in the proof of the analytic variant of Theorem \ref{thm:411}.

\bigskip \noindent
{\em Acknowledgments}. I wish to thank Reinhold H\"ubl, Rolf K\"allstr\"om and
Pierre Schapira for their useful advice in preparing this article.

\section{Adically Complete Modules} \label{sec:ad-comp}

In this section we study adic completions of modules. Theorem \ref{thm:416} is
a slight improvement of \cite[Proposition A1]{Og}.

Let $A$ be a nonzero commutative ring, and let $M$ be an $A$-module.
We do not assume that $A$ is noetherian, nor that $M$ is finitely generated..

Suppose $\{ F^i(M) \}_{i \in I}$ is an inverse system  of $A$-submodules of $M$,
indexed by a directed set $(I, \leq)$. We call such a system a {\em filtration
of $M$}. The collection of quotients \lb
$\{ M / F^i(M) \}_{i \in I}$
is also an inverse system, and its inverse limit
\begin{equation} \label{eqn:415}
\La_F(M) := \lim_{\lto i} \msp M / F^i(M)
\end{equation}
is called the {\em $F$-adic completion} of $M$.
There is a canonical homomorphism
$\tau_{M, F} : M \to \La_F(M)$,
and we say that $M$ is  {\em $F$-adically complete} if $\tau_{M, F}$ is an
isomorphism.

The next proposition is probably known, but we are not aware of a published
reference.

\begin{prop} \label{prop:420}
Let $\{ M_i \}_{i \in I}$ be an inverse system of $A$-modules, indexed by a
directed set $(I, \leq)$, and let
$M := \lim_{\lto i} \msp M_i$.
Define the filtration $\{ F^i(M) \}_{i \in I}$ of $M$
by $F^i(M) := \opn{Ker} \bigl( M \to M_i \bigr)$.
Then the $A$-module $M$ is $F$-adically complete.
\end{prop}

We do not assume that the indexing set $I$ is countable.

\begin{proof}
For every $i$ let
$N_i := M / F^i(M)$. There are injections
$\psi_i : N_i \inj M_i$, and in the limit we obtain an injection
\[ \psi := \lim_{\lto i} \psi_i := \La_F(M) =
\lim_{\lto i} N_i \inj \lim_{\lto i} M_i = M . \]
It fits in the commutative diagram of $A$-modules
\[ \begin{tikzcd} [column sep = 10ex, row sep = 6ex]
M
\ar[dr, "{\opn{id}_{M}}"]
\ar[d, "{\tau_{M, F}}"']
\\
\La_{F}(M)
\ar[r, tail, "{\psi}"']
&
M
\end{tikzcd} \]
It follows that $\psi$ is surjective, and thus it is
bijective. Therefore $\tau_{M, F}$ is bijective.
\end{proof}

\begin{cor} \label{cor:415}
Let $M$ be an $A$-module, with a filtration
$\{ F^i(M) \}_{i \in I}$ indexed by a directed set $(I, \leq)$.
Define $\wh{M} := \La_F(M)$, and put on $\wh{M}$ the filtration
$\{ F^i(\wh{M}) \}_{i \in I}$, where
$F^i(\wh{M}) := \opn{Ker} \bigl( \wh{M} \to M / F^i(M) \bigr)$.
Then $\wh{M}$ is $F$-adically complete.
\end{cor}

\begin{proof}
For every $i$ let $M_i := M / F^i(M)$.
Then $\wh{M} = \lim_{\lto i \in I} \msp M_i$,
and its filtration $F$ is the same as in Proposition \ref{prop:420}.
The proposition says that $\wh{M}$ is $F$-adically complete.
\end{proof}

An important source of filtrations comes from ideals $\a \sub A$.
Namely, given an ideal $\a$, we define
$F^i(M) := \a^i \cd M$. This is the {\em $\a$-adic filtration} of $M$,
indexed by $I = \N$, and the completion
$\La_{\a}(M) := \La_F(M)$ is called the {\em $\a$-adic completion} of $M$.
Observe that
$\La_{\a} : \cat{Mod}(A) \to \cat{Mod}(A)$
is an $A$-linear functor, and $\tau_{\a} : \opn{Id} \to \La_{\a}$ is a morphism
of functors.

We know (see \cite[Example 1.8]{Ye3}) there there are $A$-modules $M$ whose
$\a$-adic completion $\wh{M}$ is {\em not} $\a$-adically complete.
How can this be reconciled with Corollary \ref{cor:415}~?
The answer is that there are two, possibly distinct, filtrations on
$\wh{M}$, namely
$\opn{Ker} \bigl( \wh{M} \to M \msp  / \msp \a^i \cd M \bigr)$ and
$\a^i \cd \wh{M}$.
When the ideal $\a$ is finitely generated these two filtrations agree. Indeed:

\begin{prop} \label{prop:421}
Let $\a$ be a finitely generated ideal in the ring $A$, and let $M$ be an
arbitrary $A$-module. Consider the $\a$-adic filtration
$\{ \a^i \cd M \}_{i \in \N}$ on $M$, and the completion
$\wh{M} := \La_{\a}(M)$. Then:
\begin{enumerate}
\item The $A$-module $\wh{M}$ is $\a$-adically complete.

\item For every $i$ there is equality
$\a^i \cd M = \opn{Ker}( \wh{M} \to M \msp / \msp \a^i \cd M \bigr)$.
\end{enumerate}
\end{prop}

\begin{proof}
According to \cite[Corollary 3.6]{Ye3}, the $A$-module $\wh{M}$ is $\a$-adically
complete. By \cite[Theorem 1.2]{Ye3} the equality in item (2) holds.
\end{proof}

Item (1) of the proposition says that the endofunctor $\La_{\a}$ on
$\cat{Mod}(A)$ is idempotent.

The next theorem is an improvement of \cite[Proposition A1]{Og}, where
the indexing set $I$ was assumed to be countable. We give a new proof, relying
on Proposition \ref{prop:421}.

\begin{thm}[Ogus] \label{thm:416}
Let $\a$ be a finitely generated ideal in the ring $A$, and let $M$ be an
arbitrary $A$-module, with a filtration
$\{ F^i(M) \}_{i \in I}$ indexed by some directed set $(I, \leq)$.
Assume that $M$ is $F$-adically complete, and that the $F$-adic topology on $M$
is weaker than the $\a$-adic topology on it. Then $M$ is $\a$-adically complete.
\end{thm}

\begin{proof}
Since the $F$-adic topology is weaker than the $\a$-adic topology, for any
$i \in I$ there is some $j \in \N$ such that
$\a^j \cd M \sub F^i(M)$. Composing the canonical surjections
$\La_{\a}(M) \msp \to \msp M / \a^j \cd M$ and
$M \msp / \msp \a^j \cd M \to M \msp / \msp F^i(M)$
we get a homomorphism
$\La_{\a}(M) \to M / F^i(M)$.
These homomorphisms assemble into an inverse system indexed by $I$, and they
induce a homomorphism
$\psi : \La_{\a}(M) \to \La_{F}(M)$.
This homomorphism fits into the commutative diagram
\[ \begin{tikzcd} [column sep = 10ex, row sep = 6ex]
M
\ar[d, "{\tau_{M, \a}}"']
\ar[dr, "{\tau_{M, F}}", "{\simeq}"']
\\
\La_{\a}(M)
\ar[r, "{\psi}"']
&
\La_{F}(M)
\end{tikzcd} \]
We see that $\tau_{M, \a}$ is a split injection, so there is an
$A$-module isomorphism
$\La_{\a}(M) \cong M \oplus K$, where $K := \opn{Ker}(\psi)$.

Since $\La_{\a}(M)$ is $\a$-adically complete (by Proposition \ref{prop:421}),
and since $\La_{\a}$ is an $A$-linear functor, the direct summand $M$ of
$\La_{\a}(M)$ is also $\a$-adically complete.
(So in fact $\psi$ and $\tau_{M, \a}$ are bijective, and $K = 0$.)
\end{proof}

\begin{cor} \label{cor:420}
With $\a \sub A$ as in the theorem, suppose $\b \sub A$ is an ideal such that
$\a \sub \b$. If $M$ is $\b$-adically complete, then it is $\a$-adically
complete.
\end{cor}

\begin{proof}
This is because the $\b$-adic topology on $M$ is weaker than the $\a$-adic
topology.
\end{proof}

The next is \cite[Corollary A3]{Og}.

\begin{cor} \label{cor:421}
With $\a \sub A$ as in the theorem, assume $M$ is $\a$-adically complete, and
$N \sub M$ is an $A$-submodule that is closed for the $\a$-adic topology. Then
$N$ is $\a$-adically complete.
\end{cor}

\begin{proof}
Take any submodule $N$ of $M$. Consider the filtration
$\{ F^i(N) \}_{i \in \N}$ of $N$, $F^i(N) := N \cap \a^i \cd M$.
An element $m \in M$ is in the closure of $N$,
for the $\a$-adic topology, if and only if for every $i$ the image of $m$ in
$M / \a^i \cd M$ belongs to the submodule
$N / F^i(N)$. This means that the closure of $N$ inside
$M = \lim_{\lto i} \msp M / \a^i \cd M$ is
$\lim_{\lto i} \msp N / F^i(N)$.

We are given that the submodule $N$ is closed. Hence
$N = \lim_{\lto i} \msp N / F^i(N)$, i.e.\ it is $F$-adically
complete. Since the $F$-adic topology on $N$ is weaker than its $\a$-adic
topology, we conclude that $N$ is $\a$-adically complete.
\end{proof}

The next is a slight improvement of \cite[Corollary A2]{Og};
again we do not require the set $I$ to be countable.

\begin{cor} \label{cor:422}
With $\a \sub A$ as in the theorem, let
$\{ M_i \}_{i \in I}$ be an inverse system of $A$-modules,
indexed by a directed set $(I, \leq)$.
Assume every $M_i$ is an $\a$-adically complete module. Define the $A$-module
$M := \lim_{\lto i \in I} M_i$. Then $M$ is $\a$-adically complete.
\end{cor}

\begin{proof}
Consider the set $K := I \times \N$, with this quasi-order:
$k_1 = (i_1, j_1) \leq k_2 = (i_2, j_2)$ if
$i_1 \leq i_2$ and $j_1 \leq j_2$. So $(K, \leq)$ is a directed set.
For $k = (i, j) \in K$ define the $A$-module
$M_k := M_i \msp / \msp \a^j \cd M_i$.
We know that
$M_i \cong \lim_{\lto j \in \N} M_{i, j}$ and
$M = \lim_{\lto i \in I} M_{i}$.
By the associativity property of inverse limits, we get
$M \cong  \lim_{\lto k \in K} M_{k}$.
Letting $F^k(M) := \opn{Ker} \bigl( M \to M_k \bigr)$,
Proposition \ref{prop:420} tells us that $M$ is $F$-adically complete.

The $F$-adic topology is weaker that the $\a$-adic topology on $M$, because
$\a^j \cd M \sub \lb F^{(i, j)}(M)$ for all
$i \in i$ and $j \in \N$. According to Theorem \ref{thm:416}, $M$ is
$\a$-adically complete.
\end{proof}

\section{Rings of Differential Operators}
\label{sec:rings-DOs}

In this section we discuss differential operators between adically complete
rings and modules.

We begin by recalling material from \cite[Section 16]{EGA-IV}.
Let $A \to B$ be a homomorphism between nonzero commutative rings.
Given $B$-modules $M$ and $N$, consider the $A$-central $B$-bimodule
$\opn{Hom}_{A}(M, N)$, where the action by $B$ is
$(b_1 \cd \th \cd b_2)(m) := b_1 \cd \th(b_2 \cd m)$
for $b_i \in B$, $m \in M$ and $\th \in \opn{Hom}_{A}(M, N)$.
We refer to the homomorphism $\th$ as an {\em operator}.
The operator $\th$ is said to have {\em differential order} $\leq 0$
if it is $B$-linear. For $k \geq 1$ the definition is inductive:
an operator $\th \in \opn{Hom}_{A}(M, N)$ has
differential order $\leq k$ if for every $b \in B$ the operator
$[b, \th] := b \cd \th - \th \cd b$ has
differential order $\leq k - 1$.
The set of all differential operators of order $\leq k$ is the $A$-submodule
$\opn{Diff}^k_{B / A}(M, N)$ of $\opn{Hom}_{A}(M, N)$, and the set of all
differential operators is
\begin{equation} \label{eqn:440}
\opn{Diff}_{B / A}(M, N) := \bigcup_{k \geq 0} \opn{Diff}^k_{B / A}(M, N)
\sub \opn{Hom}_{A}(M, N)  .
\end{equation}

It is not hard to see that for $B$-modules $M_0$, $M_1$ and $M_2$, composition
is a function
\[  \opn{Diff}^{k_1}_{B / A}(M_0, M_1) \times
\opn{Diff}^{k_2}_{B / A}(M_1, M_2) \to
\opn{Diff}^{k_1 + k_2}_{B / A}(M_0, M_2) . \]
Thus for a single module $M$,
$\opn{Diff}^{}_{B / A}(M, M)$ is a filtered $B$-ring, central over $A$.

For $M = B$ we obtain the ring of differential operators
$\DD_{B / A} := \opn{Diff}^{}_{B / A}(B, B)$,
which comes with the filtration $\{ \DD^k_{B / A} \}_{k \in \N}$.
In this case $\DD^0_{B / A} = B$, but $B$ is not in the center of
$\DD_{B / A}$.

\begin{prop} \label{prop:442}
Let $M$ be a left $\DD_{B / A}$-module. For any operator
$\th \in \DD^k_{B / A}$, the operator
$\th_M \in \opn{End}_A(M)$, $\th_M(m) := \th \cd m$, is a
differential operator of order $\leq k$.
\end{prop}

\begin{proof}
By induction on $k$. For $k = 0$ we have $\DD^0_{B / A} = B$, so $\th_M$ has
order $\leq 0$. Now take $k \geq 1$. For every $b \in B$ and $m \in M$ we have
\[ [b, \th_M](m) = b \cd \th \cd m - \th \cd b \cd m = [b, \th]_M(m)  \]
in $M$.
Now $[b, \th] \in \DD^{k - 1}_{B / A}$, so by the induction hypothesis
$[b, \th]_M \in \opn{Diff}^{k - 1}_{B / A}(M, M)$.
We conclude that $[b, \th_M]$ is a differential operator of order $\leq k
- 1$ for all $b$. Therefore $\th_M$ is  a differential
operator of order $\leq k$.
\end{proof}

Let $I$ be the kernel of the multiplication homomorphism
$B \ot_A B \to B$. The ring
$\mcal{P}^k_{B / A} := (B \ot_A B) / I^{k + 1}$
is called the {\em module of principal parts of order $k$}.
See \cite[Subsection 16.3]{EGA-IV}.
We view $\mcal{P}^k_{B / A}$ as a left $B$-module by the ring homomorphism
$B \to \mcal{P}^k_{B / A}$, $b \mapsto b \ot 1$.
Then the ring homomorphism
$\d^k : B \to  \mcal{P}^k_{B / A}$, $b \mapsto 1 \ot b$, is a differential
operator of order $k$. It is universal in the following sense:
for $B$ modules $M$ and $N$, the $A$-module homomorphism
\begin{equation} \label{eqn:441}
\opn{Hom}_{B}(\mcal{P}^k_{B / A} \ot_B M, N) \to
\opn{Diff}^k_{B / A}(M, N)\ , \
\phi \mapsto \phi \circ \d^k_M \ ,
\end{equation}
is bijective. Here
$\d^k_M(m) := \d^k(1_B) \ot m \in \mcal{P}^k_{B / A} \ot_B M$.

\begin{lem} \label{lem:440}
Let $\b \sub B$ be an ideal, and let
$\th \in \opn{Diff}^k_{B / A}(M, N)$. Then
$\th(\b^{i + k} \cd M) \sub \b^{i} \cd N$ for all $i \in \N$.
\end{lem}

\begin{proof}
This is \cite[Proposition 1.4.6]{Ye1}.
\end{proof}

Suppose $M$ and $N$ are $B$-modules, $\b \sub B$ is a finitely generated
ideal, and $\th : M \to N$ is an $A$-linear homomorphism that's continuous for
the $\b$-adic topologies on these modules. Then $\th$ extends uniquely to a
continuous $A$-linear homomorphism
$\wh{\th} : \wh{M} \to \wh{N}$ between the $\b$-adic completions.

\begin{prop} \label{prop:440}
Let $\b \sub B$ be a finitely generated ideal, and let $\wh{B}$ be the
$\b$-adic completion of $B$. Let $M$ and $N$ be $B$-modules,
and let $\th \in \opn{Diff}^k_{B / A}(M, N)$.
\begin{enumerate}
\item The $A$-linear homomorphism $\th : M \to N$ is continuous for the
$\b$-adic topologies.

\item The unique continuous extension
$\wh{\th} : \wh{M} \to \wh{N}$ of $\th$ belongs
to $\opn{Diff}^k_{\wh{B} / A}(\wh{M}, \wh{N})$.
\end{enumerate}
\end{prop}

\begin{proof}
(1) Clear from Lemma \tup{\ref{lem:440}}.

\medskip \noindent
(2) The operator $\wh{\th}$ is a differential operator of order $\leq k$ iff
$[ \cdots [\wh{\th}, b_0], \ldots, b_k] = 0$
for all $b_0, \ldots, b_k \in \wh{B}$.
This commutator expression is a continuous multilinear function
$\wh{B}^{\msp k + 1} \to \wh{B}$, and it vanishes on the image of
$B^{\msp k + 1}$. Therefore it is zero on $\wh{B}^{\msp k + 1}$.
\end{proof}

\begin{cor} \label{cor:440}
Let $\b \sub B$ be a finitely generated ideal, and let $\wh{B}$ be the
$\b$-adic completion of $B$.
\begin{enumerate}
\item Any differential operator $\th \in \DD^k_{B / A}$
extends uniquely to a differential operator $\wh{\th} \in \DD^k_{\wh{B} / A}$.

\item The function $\th \mapsto \wh{\th}$ is a filtered $A$-ring homomorphism
$\DD_{B / A} \to \DD_{\wh{B} / A}$.
\end{enumerate}
\end{cor}

\begin{proof}
(1) This is clear from Proposition \ref{prop:440}.

\medskip \noindent
(2) This follows from the uniqueness in (1).
\end{proof}

Now assume $B = A[t_1, \ldots, t_n]$, the polynomial ring in $n$ variables.
$A$ is still an arbitrary nonzero commutative ring.
For a multi-index $\bs{i} = (i_1, \ldots, i_n) \in \N^n$
we write
$\bs{t}^{\bs{i}} := t_1^{i_1} \cdots t_n^{i_n}$
and $\abs{\bs{i}} := i_1 + \cdots + i_n$.
Then (see \cite[Section 1.4]{Ye1}) $\PP^k_{B / A}$ is a free left
$B$-module with basis the collection
$\{ \d^k(\bs{t}^{\bs{i}}) \}_{\abs{\bs{i}} \leq k}$.
From this we get the dual basis
$\{ \th_{\bs{i}} \}_{\abs{\bs{i}} \leq k}$ of $\DD^{k}_{B / A}$.
Namely, the differential operator $\th_{\bs{i}}$ is the unique differential
operator such that
$\th_{\bs{i}}(\bs{t}^{\bs{i}}) = 1$ and
$\th_{\bs{i}}(\bs{t}^{\msp \bs{j}}) = 0$ for $\bs{j} \neq \bs{i}$.
For $k = 1$ we have the partial derivatives
$\pa_i = \pa / \pa_{t_i} \in \DD^{1}_{B / A}$,
which satisfy $\pa_i(t_i) = 1$, and $\pa_i(t_j) = 0$ if $j \neq i$.

If $A$ contains $\Q$, then, due to the uniqueness above, for every
$\bs{i} = (i_1, \ldots, i_n)$ there is equality
\begin{equation} \label{eqn:442}
\th_{\bs{i}} =  (i_1 ! \cdots i_n!)^{-1} \cd
\pa_1^{i_1} \cdots \pa_n^{i_n}  \in \DD_{B / A} .
\end{equation}
It follows that as a left $B$-module, $\DD^{}_{B / A}$ is free with basis the
collection of monomial operators
$\{ \pa_1^{i_1} \cdots \pa_n^{i_n} \}_{\bs{i} \in \N^n}$.
The relations satisfied by the multiplication in $\DD^{}_{B / A}$ are these:
$\pa_i \cd \pa_j = \pa_j \cd \pa_i$, and
$\pa_i \cd b = b \cd \pa_i + \pa_i(b)$ for $b \in B$.

Consider the ideal $\b \sub B$ generated by the variables. The $\b$-adic
completion of $B$ is the ring $\wh{B} = A[[t_1, \ldots, t_n]]$ of power series.
The next proposition describes the structure of the ring
$\DD_{\wh{B} / A}$.

\begin{prop} \label{prop:425}
Let $A$ be a commutative ring containing $\Q$, let
$B := A[t_1, \ldots, t_n]$, and let
$\wh{B} := A[[t_1, \ldots, t_n]]$.
\begin{enumerate}
\item Any differential operator $\th \in \DD_{B / A}$ of order $\leq k$
extends uniquely to a  differential operator
$\wh{\th} \in \DD_{\wh{B} / A}$ of order $\leq k$.

\item The function $\th \mapsto \wh{\th}$ is a filtered ring homomorphism
$\DD_{B / A} \to \DD_{\wh{B} / A}$.

\item The left $\wh{B}$-module homomorphism
$\wh{B} \ot_{B} \DD_{B / A} \to \DD_{\wh{B} / A}$ is bijective.

\item $\DD_{\wh{B} / A}$ is a free left $\wh{B}$-module with basis
$\{ \pa_1^{i_1} \cdots \pa_n^{i_n} \}_{\bs{i} \in \N^n}$.

\item The relations satisfied by the multiplication in $\DD^{}_{\wh{B} / A}$
are these:
$\pa_i \cd \pa_j = \pa_j \cd \pa_i$, and
$\pa_i \cd b = b \cd \pa_i + \pa_i(b)$ for $b \in \wh{B}$.
\end{enumerate}
\end{prop}

\begin{proof}
(1-2) These are special cases of items (1-2) of Corollary \ref{cor:440}.

\medskip \noindent (3)
By item (1) there is a canonical $A$-module isomorphism
\begin{equation} \label{eqn:450}
\DD_{\wh{B} / A} = \opn{Diff}^k_{\wh{B} / A}(\wh{B}, \wh{B}) \cong
\opn{Diff}^k_{B / A}(B, \wh{B}) .
\end{equation}
For every $k$, the left $B$-module
$\PP^k_{B / A}$ is free of finite rank. Therefore, using (\ref{eqn:450}),
we have isomorphisms of $A$-modules
\[ \DD^k_{\wh{B} / A} \cong \opn{Hom}_{B}(\mcal{P}^k_{B / A}, \wh{B})
\cong \wh{B} \ot_B \opn{Hom}_{B}(\mcal{P}^k_{B / A}, B) \cong
\wh{B} \ot_B \DD^k_{B / A} . \]
The composed isomorphism
$\wh{B} \ot_B \DD^k_{B / A} \iso \DD^k_{\wh{B} / A}$
is compatible with the ring homomorphism
$\DD_{B / A} \to \DD_{\wh{B} / A}$ from item (1).

\medskip \noindent (4)
This is clear from item (3) and formula (\ref{eqn:442}).

\medskip \noindent (5)
These relations hold in $\DD_{B / A}$, so by continuity they also hold in
$\DD_{\wh{B} / A}$.
\end{proof}

\begin{cor} \label{cor:425}
In the setting of Proposition \tup{\ref{prop:425}}, let $M$ be a left
$\DD_{B / A}$-module, and let $\wh{M}$ be the $\b$-adic completion of $M$.
Then $\wh{M}$ has a unique left $\DD^{}_{\wh{B} / A}$-module structure, such
that $\tau_M : M \to \wh{M}$ is $\DD_{B / A}$-linear.
\end{cor}

\begin{proof}
Of course $\wh{M}$ is a $\wh{B}$-module.
By  Propositions \ref{prop:442} and \ref{prop:440}, the partial derivatives
$\pa_i$ act continuously on $M$, and hence they extend uniquely to
continuous $A$-linear endomorphisms on $\wh{M}$. By continuity, the relations
in Proposition \ref{prop:425}(5) are satisfied in
$\opn{End}^{\mrm{cont}}_A(M)$. So we get a left
$\DD^{}_{\wh{B} / A}$-module structure on $M$.
\end{proof}

\begin{rem} \label{rem:440}
Let $A \to B$ be a ring homomorphism, $\b \sub B$ a finitely generated ideal,
and $\wh{B}$ the $\b$-adic completion of $B$. We can give a description of
the ring $\DD_{\wh{B} / A}$ as the {\em dir-inv completion} of the ring
$\DD_{B / A}$.

Dir-inv structures where introduced in \cite[Section 1]{Ye2}. In our setting
we put on $\DD_{B / A}$ the dir structure (an ascending filtration)
$\{ \DD^k_{B / A} \}_{k \in \N}$, and on each $\DD^k_{B / A}$ we put the inv
structure (a descending filtration)
$\{ F^i(\DD^k_{B / A}) \}_{i \in \Z}$, where
\[ F^i (\DD^k_{B / A}) :=
\bigl\{ \th \in \DD^k_{B / A} \mid
\th(\b^j) \sub \b^{\opn{max}(j + i, \lsp 0)} \ \
\tup{for all} \ \ j \in \N \bigr\} . \]
The dir-inv completion of $\DD_{B / A}$ is
\[ \La_{\mrm{di}}(\DD_{B / A}) :=
\lim_{k \to} \ \lim_{\lto i} \msp
\DD^k_{B / A} \msp / \msp F^i(\DD^k_{B / A}) .  \]
This completion has a unique $A$-ring structure such that
the diagram of $A$-rings
\[ \begin{tikzcd} [column sep = 10ex, row sep = 6ex]
\DD_{B / A}
\ar[dr]
\ar[d]
\\
\DD_{\wh{B} / A}
\ar[r, "{\simeq}"]
&
\La_{\mrm{di}}(\DD_{B / A})
\end{tikzcd} \]
is commutative, and the horizontal arrow is an isomorphism.

This description of the ring $\DD_{\wh{B} / A}$ can be used to prove a
generalization of Corollary \ref{cor:425} to the setting of this remark.
\end{rem}

\section{Complete \texorpdfstring{$\DD$}{D}-Modules -- the One Dimensional
Case}
\label{sec:one-dim}

This section is devoted to the proof of Theorem \ref{thm:240},
which is the $1$-dimensional case of Theorem \ref{thm:411} from the
introduction. We give a rather detailed proof; but the essential idea
already appears in the prior papers of Ogus and others, see Remark
\ref{rem:460}.

Throughout this section we assume the following:

\begin{setup} \label{set:226}
We are given a commutative ring $A$ containing $\Q$.
Let $B := A[[t]]$, the ring of power series in the variable $t$, and let
$\b \sub B$ be the ideal generated by $t$.
We are also given a left $\DD_{B / A}$-module $M$, which is $\b$-adically
complete as a $B$-module.
\end{setup}

Observe that the ring $A$ is not assumed to be noetherian, and the $B$-module
$M$ is not assumed to be finitely generated.

By the definition of the $\b$-adic topology on $M$, the submodules
$\b^i \cd M$ are open in $M$, and thus -- as always in linearly topologized
modules -- also closed.

As explained in Section \ref{sec:rings-DOs},
the ring of differential operators $\DD_{B / A}$ is the subring
of $\opn{End}_{A}(B)$ generated by $B$ and the derivation
$\pa  := \smfrac{\partial}{\partial t}$.
An operator $\phi \in \DD_{B / A}$ of differential order $\leq i$ has a unique
sum expansion
\begin{equation} \label{eqn:190}
\phi = \sum_{0 \leq j \leq i} \ b_{j} \cd \pa^{\msp j}
\end{equation}
in $\DD_{B / A}$, with coefficients $b_{j} \in B$.
According to Proposition \ref{prop:425}(5), the relation satisfied by the
element $\pa$ in the noncommutative ring $\DD_{B / A}$ is
\begin{equation} \label{eqn:460}
\pa \cd b = b \cd \pa + \pa(b) ,
\end{equation}
where $\pa(b) \in B$ is the result of action
of the operator $\pa$ on an element $b \in B$.
For every $m \in M$ and $b \in B$ we have the equality
\begin{equation} \label{eqn:285}
\pa(b \cd m) = (\pa \cd b) \cd m =
(b \cd \pa + \pa(b)) \cd m = b \cd \pa(m) + \pa(b) \cd m.
\end{equation}
This formula implies that
$\pa(\b^i \cd M) \sub \b^{i - 1} \cd M$, so $\pa$ acts continuously on $M$
for its $\b$-adic topology. By formula (\ref{eqn:190}) we see that every
operator $\phi \in \DD_{B / A}$ acts continuously on $M$.

\begin{dfn} \label{dfn:460} \mbox{}
\begin{enumerate}
\item For every $i \in \N$ define the order $i$ differential operator
\[ \psi_i := (i!)^{-1} \cd (-t)^{i} \cd \pa^i \in \DD_{B / A} . \]

\item For every element $m \in M$ define the element
\[ \psi(m) := \sum_{i \in \N} \msp \psi_i(m) =
\sum_{i \in \N} \msp (i!)^{-1} \cd (-t)^{i} \cd \pa^i(m) \in M .  \]
\end{enumerate}
\end{dfn}

Since $\psi_i(m) \in \b^i \cd M$, the sum in item (2) converges in the
$\b$-adically complete module $M$.
See Remark \ref{rem:271} regarding the possible convergence of the operator sum,
namely making sense of the formula
$\psi = \sum_{i \in \N} \msp \psi_i$.

Recall that the module of horizontal elements of $M$ is
$M^{\mrm{hor}} = \{ m \in M \mid \pa(m) = 0 \}$.

\begin{lem} \label{lem:275}
Under Setup \tup{\ref{set:226}}, the following hold for the function
$\psi : M \to  M$ from Definition \tup{\ref{dfn:460}(2)}.
\begin{enumerate}
\item The function $\psi$ is $A$-linear.

\item $\psi(m) \in M^{\mrm{hor}}$ for every $m \in M$.

\item If $m \in M^{\mrm{hor}}$ then $\psi(m) = m$.

\item $m - \psi(m) \in \b \cd M$ for every $m \in M$.
\end{enumerate}
\end{lem}

\begin{proof}
(1) This is because each $\psi_i$ is $A$-linear, and by the pointwise
convergence of the sum
$\psi(m) = \sum_{i \in \N} \msp \psi_i(m)$ in $M$.

\medskip \noindent
(2) Due to the continuity of $\pa$, and using formula (\ref{eqn:285}),
for every $m \in M$ we have
\[ \begin{aligned}
&
\pa(\psi(m)) =
\sum_{i \in \N} \msp \pa \bigl( (i!)^{-1} \cd (-t)^{i} \cd \pa^i(m) \bigr)
\\ &
\qquad = \sum_{i \in \N} \msp
\Bigl( (i!)^{-1} \cd (-t)^{i} \cd \pa^{i + 1}(m) -
(i!)^{-1} \cd i \cd (-t)^{i - 1} \cd \pa^{i}(m) \Bigl)
\\ &
\qquad = \Bigl( \sum_{i \in \N} \msp (i!)^{-1} \cd (-t)^{i} \cd \pa^{i + 1}(m)
\Bigr)
- \Bigl( \sum_{j \in \N} \msp (j!)^{-1} \cd (-t)^{j} \cd \pa^{j + 1}(m) \Bigr)
= 0 .
\end{aligned} \]

\medskip \noindent
(3) Take some $m \in M^{\mrm{hor}}$. Then
$\psi_0(m) = m$, and $\psi_i(m) = 0$ for $i \geq 1$.
Thus $\psi(m) = m$.

\medskip \noindent
(4) Take some element $m \in M$. By definition we have $\psi_0(m) = m$,
and $\psi_i(m) \in \b^i \cd M \sub \b \cd M$ for $i \geq 1$.
Since $\b \cd M$ is closed, we have
$\psi(m) - m = \sum_{i \geq 1}  \psi_i(m) \in \b \cd M$.
\end{proof}

It is not immediately clear that $\psi : M \to M$ is continuous. One can show
it by a messy calculation; but it will be an easy consequence later, see
Corollary \ref{cor:460} below.

Define the $B$-module $\widebar{M} := M \msp / \msp  \b \cd M$, with
canonical projection $\pi : M \surj \widebar{M}$.

\begin{prop}[Existence of a Solution] \label{prop:430}
Under Setup \tup{\ref{set:226}}, given an element
$\widebar{m} \in \widebar{M}$, there exists an element $m \in M$
such that $\pa(m) = 0$ and $\pi(m) = \widebar{m}$.
\end{prop}

Later, in Corollary \ref{cor:296}, we shall see that the element $m$ is
unique.

\begin{proof}
Let $m'$ be some lifting of $\widebar{m}$ to $M$; namely $m' \in M$ satisfies
$\pi(m') = \widebar{m}$. Define $m := \psi(m') \in M$, where
$\psi$ is the operator from Definition \ref{dfn:460}(2).
According to Lemma \ref{lem:275}(2) the element $m$ satisfies
$\pa(m) = 0$, and by  Lemma \ref{lem:275}(4) we have
$\pi(m) = \pi(m') = \widebar{m}$.
\end{proof}

Here is a result on induced $\DD_{B / A}$-modules.
Let $L$ be an $A$-module.
The $B$-module $B \ot_A L$ has an induced left $\DD_{B / A}$-module
structure, with action $\pa(b \ot l) = \pa(b) \ot l$
for $b \in B$ and $l \in L$.
Now define the $B$-module
$N := B \msp \wh{\ot}_A \msp L$, the $\b$-adic completion of
$B \ot_A L$ (where $L$ is given the discrete topology).
By continuity, the $B$-module $N$ has a left $\DD_{B / A}$-module
structure.

\begin{lem} \label{lem:281}
For the $B$-module $N = B \msp \wh{\ot}_A \msp L$, with its
induced $\DD_{B / A}$-module structure, the following assertions hold.
\begin{enumerate}
\item Every element $n \in N$ has a unique convergent power series expansion
$n = \lb \sum_{i \geq 0} \, t^i \ot l_i$, with $l_i \in L$.

\item For $n \in N$, with its power series expansion from item \tup{(1)},
there is equality
$\pa(n) = \sum_{i \geq 1} \, i \cd t^{i - 1} \ot l_i$ in $N$.

\item The $A$-module homomorphism
$\al : L \to  N^{\mrm{hor}}$, $\al(l) := 1_B \ot l$, is bijective.
\end{enumerate}
\end{lem}

\begin{proof}
(1) For $i \in \N$ let $B_i := B / \b^{i + 1}$ and
$N_i := N \msp / \msp \b^{i + 1} \cd N$. Then
$B_i \cong \boplus_{0 \leq j \leq i} \, A \cd t^j$ and
\begin{equation} \label{eqn:287}
N_i \cong B_i \ot_{B} N \cong  B_i \ot_{A} L \cong
\boplus_{0 \leq j \leq i} \, t^j \ot L
\end{equation}
as $A$-modules, and these isomorphisms are compatible with the change of $i$.
The image of $n$ in $N_i$ is uniquely a sum
$\sum_{0 \leq j \leq i} \, t^j \ot l_j$
with $l_j \in L$. In the inverse limit
$N \cong \lim_{\leftarrow i} \, N_i$ we obtain
$n = \sum_{j \geq 0} \, t^j \ot l_j$.

\medskip \noindent
(2) For each $i$ the derivation $\pa$ acts on $t^i \ot l_i$ as follows:
$\pa(t^i \ot l_i) = \pa(t^i) \ot l_i$.
Also $\pa(t^i) = i \cd t^{i - 1}$. By $A$-linearity and continuity we
get the stated expression for $\pa(n)$.

\medskip \noindent
(3) Items (1)-(2) say that $n \in N^{\mrm{hor}}$ iff $l_i = 0$ for all
$i \geq 1$, namely iff $n = t^0 \ot l_0 = \al(l_0)$.
\end{proof}

Now we take the $A$-module
$L := M^{\mrm{hor}}$.
This gives us the complete $B$-module
$N := B \msp \wh{\ot}_A \msp M^{\mrm{hor}}$ and the $B$-module homomorphism
\begin{equation} \label{eqn:430}
\th : N = B \msp \wh{\ot}_A \msp M^{\mrm{hor}} \to M .
\end{equation}

\begin{lem} \label{lem:201}
Consider $N = B \msp \wh{\ot}_A \msp M^{\mrm{hor}}$ as above, with the induced
left $\DD_{B / A}$-module structure.
\begin{enumerate}
\item The $B$-module homomorphism
$\th : N \to M$ is a homomorphism of left $\DD_{B / A}$-modules.

\item The $A$-module homomorphism
$\th^{\mrm{hor}} := \th|_{N^{\mrm{hor}}} : N^{\mrm{hor}} \to
M^{\mrm{hor}}$
is bijective.
\end{enumerate}
\end{lem}

\begin{proof}
(1) It suffices to show that $\th$ respects the action of the derivation
$\pa$. Let's denote by $\pa_M$ and $\pa_N$ the action of $\pa$ on the two
modules.
Take $n \in N$, with power series expansion
$n = \sum_{i \geq 0} \, t^i \ot l_i$, $l_i \in L = M^{\mrm{hor}}$,
see Lemma \ref{lem:281}(1). Since $\th$, $\pa_M$ and $\pa_N$ are continuous,
and $\pa(l_i) = 0$, we get
\[ \th(\pa_N(n)) =
\th \Bigl( \sum_{i \geq 1} \, i \cd t^{i - 1} \ot l_i \Bigr)
= \sum_{l \geq 1} \, \th(i \cd t^{i - 1} \ot l_i) =
\sum_{i \geq 1} \, i \cd t^{i - 1} \cd l_i  \]
and
\[ \pa_N(\th(n)) = \pa_N \Bigl( \sum_{i \geq 0} \, t^{i} \cd l_i \Bigr)
= \sum_{i \geq 1} \, i \cd t^{i - 1} \cd l_i  \ . \]
We see that
$\th \circ \pa_N = \pa_M \circ \th$.

\medskip \noindent
(2) According to Lemma \ref{lem:281}(3) the homomorphism
$\al : L \to N^{\mrm{hor}}$ is bijective.
And $\th^{\mrm{hor}} \circ \al = \opn{id}_L$.
Hence $\th^{\mrm{hor}}$ is
bijective.
\end{proof}

We have the inclusion
$\ep_M : M^{\mrm{hor}} \to M$,
the projection
$\pi_M : M \to \widebar{M} = M \msp / \msp \b \cd M$, and
their composition
$\ga_M := \pi_M \circ \ep_M$.
There are similar homomorphisms for $N$.
These sit in the following commutative diagram of $A$-modules.

\begin{equation} \label{eqn:200}
\begin{tikzcd} [column sep = 9ex, row sep = 6ex]
L
\ar[r, "{\al}", "{\simeq}"']
\ar[rr, bend left = 30, start anchor = north, end anchor = north,
"{\opn{id}_L}", "{\simeq}"']
&
N^{\mrm{hor}}
\ar[r, "{\th^{\mrm{hor}}}", "{\simeq}"']
\ar[d, rightarrowtail, "{\ep_N}"]
\ar[dd, bend right = 35, start anchor = south west, end anchor = north west,
"{\ga_N}"', "{}"]
&
M^{\mrm{hor}} = L
\ar[d, rightarrowtail, "{\ep_M}"']
\ar[dd, bend left = 35, start anchor = south east, end anchor = north east,
"{\ga_M}"]
\\
&
N
\ar[r, "{\th}"]
\ar[d, twoheadrightarrow, "{\pi_N}"]
&
M
\ar[d, twoheadrightarrow, "{\pi_M}"']
\\
&
\widebar{N}
\ar[r, "{\widebar{\th}}"]
&
\widebar{M}
\end{tikzcd}
\end{equation}

Here is the main result of this section.

\begin{thm} \label{thm:240}
Under Setup \tup{\ref{set:226}}, the $\DD_{B / A}$-module homomorphism
\[ \th : B \msp \wh{\ot}_A \msp  M^{\mrm{hor}} \to M \]
is bijective.
\end{thm}

\begin{proof}
The proof is divided into two steps.
As before we write
$N := B \msp \wh{\ot}_A \msp  M^{\mrm{hor}}$.

\smallskip \noindent
Step 1. In this step we prove that $\th : N \to M$ is surjective.

According to Proposition \ref{prop:430}, the homomorphism
$\ga_M : M^{\mrm{hor}} \to \widebar{M}$ is surjective.
By Lemma \ref{lem:201}(2) we know that
$\th^{\mrm{hor}} : N^{\mrm{hor}} \to M^{\mrm{hor}}$
is bijective. It follows that the homomorphism
$\pi_M \circ \th : N \to \widebar{M}$
is surjective. Since both $B$-modules $M$ and $N$ are $\b$-adically complete,
the Complete Nakayama Theorem \cite[Theorem 1.10]{Ye4} says that
$\th : N \to M$ is surjective.

\medskip \noindent
Step 2. In this step we prove that $\th : N \to M$ is injective, and hence it
is bijective.

Assume, for the sake of arriving at a contradiction, that $\th : N \to M$ is
not injective. Then there exists some nonzero element $n \in N$ such that
$\th(n) = 0$. The element $n$ has a unique convergent power series
expansion
$n = \sum_{i \geq 0} \, t^i \ot l_i$,
with $l_i \in L = M^{\mrm{hor}}$; see Lemma \ref{lem:281}(1).

Since $n$ is nonzero, its $\b$-adic order is finite, say $i_0 \in \N$. So in
the power series expansion of $n$ we have
$l_i = 0$ for all $i < i_0$, and $l_{i_0} \neq 0$.
Define $n' = \pa^{i_0}(n) \in N$.
Let $n' = \sum_{i \geq 0} \, t^i \cd l'_i$ be the power series expansion of
$n'$, with $l'_i \in L$. According to Lemma \ref{lem:281}(2)
we have $l'_0 = i_0 ! \cd l_{i_0}$, which is a nonzero element of
$L = M^{\mrm{hor}}$.
This means that $n' \notin \b \cd N$.

Let $\psi_N$ be the idempotent operator from Definition \ref{dfn:460}(2),
but for the complete $\DD_{B / A}$-module $N$.
Define the element
$n'' := \psi_N(n') \in N^{\mrm{hor}}$.
By Lemma \ref{lem:275}(4) we have $n' - n'' \in  \b \cd N$,
so $n'' \notin \b \cd N$, and hence $n''$ is a nonzero element of
$N^{\mrm{hor}}$.
By Lemma \ref{lem:201}(2) the element
$\th(n'') = \th^{\mrm{hor}}(n'') \in M^{\mrm{hor}}$
is nonzero.

On the other hand, we are assuming that $\th(n) \in M$ is zero. Since the
homomorphism $\th$ is $\DD_{B / A}$-linear, see Lemma \ref{lem:201}(1), we
get
$\th(n') = \th(\pa^{i_0}(n)) = \pa^{i_0}(\th(n)) = 0$.
For every $j$ the operator $\psi_j$ belongs to $\DD_{B / A}$, and thus
$\th(\psi_j(n')) = \psi_j(\th(n')) = 0$.
But $\th$ is continuous, so
\[ \th(n'')  = \th(\psi_N(n')) =
\th \Bigl( \sum\nolimits_{j \geq 0} \, \psi_j(n') \Bigr)
= \sum\nolimits_{j \geq 0} \, \th(\psi_j(n')) = 0 . \]
This contradicts the calculation in the previous paragraph.
\end{proof}

\begin{cor} \label{cor:460}
Under Setup \tup{\ref{set:226}}, the following hold:
\begin{enumerate}
\item There is an $A$-linear direct sum decomposition
$M = M^{\mrm{hor}} \oplus \b \cd M$.

\item The homomorphism $\psi : M \to M$ is the $A$-linear projection
onto $M^{\mrm{hor}}$ in the decomposition in item \tup{(1)}.
In particular, $\opn{Ker}(\psi) = \b \cd M$.

\item The homomorphism $\psi : M \to M$ is continuous for the $\b$-adic topology
of $M$.
\end{enumerate}
\end{cor}

\begin{proof}
(1) The $A$-module $B$ decomposes into $B = A \oplus \b$.
This gives a decomposition of $N := B \msp \wh{\ot}_A \msp  M^{\mrm{hor}}$ into
$N = M^{\mrm{hor}} \oplus \b \msp \wh{\ot}_A \msp M^{\mrm{hor}}$.
The isomorphism $\th$ in the theorem carries this direct sum decomposition to
$M$.

\medskip \noindent
(2) By Lemma \ref{lem:275}(2-3) we know that $\psi$ is a projection (an
idempotent operator) with image $M^{\mrm{hor}}$, and by
Lemma \ref{lem:275}(4) we know that $\opn{Ker}(\psi) \sub \b \cd M$.
This forces the equality $\opn{Ker}(\psi) = \b \cd M$.

\medskip \noindent
(3) The kernel of $\psi$ is the open $A$-submodule $\b \cd M$.
\end{proof}

\begin{rem} \label{rem:460}
The projection operator $\psi$ that we use here is the same as the operator
called $P$ in formula (8.9.5) of \cite{Ka}, and in formula (3) of \cite{Og}.
We thank R. K\"{a}llstrom for clarifying this to us.
\end{rem}

\begin{rem} \label{rem:271}
By definition, for every $m \in M$ we have
$\psi(m) = \sum_{i \geq } \psi_i(m)$.
Each $\psi_i$ is a differential operator; but the operator $\psi$
is merely continuous.

It seems plausible that after endowing the ring
$\opn{End}^{\mrm{cont}}_{A}(M)$
of continuous $A$-linear endomorphisms of $M$ with a suitable topology,
there will be equality
$\psi = \sum_{i \geq 0} \, \psi_i$
inside $\opn{End}^{\mrm{cont}}_{A}(M)$.
Moreover, if such a topology can be produced, then the topological closure
$\what{\DD}_{B / A}$ of
$\DD_{B / A}$ inside $\opn{End}^{\mrm{cont}}_{A}(M)$
will be some sort of ring of "pro-differential operators", containing
the projection operator $\psi$.
It is conceivable that the suitable topology on
$\opn{End}^{\mrm{cont}}_{A}(M)$
will make it into a {\em semi-topo\-logical ring}, in the sense of
\cite[Section 1.2]{Ye1}.
\end{rem}

\section{Complete \texorpdfstring{$\DD$}{D}-Modules -- the High Dimensional
Case}
\label{sec:high-dim}

In this section we prove Theorem \ref{thm:215}, which is Theorem \ref{thm:411}
in the Introduction.
As already mentioned, it is a rephrasing, and an improvement to the
non-noetherian setting, of \cite[Theorem 1.3]{Og}; which itself is an
improvement of \cite[Proposition 8.9]{Ka}.

Let us recall a few facts from Sections \ref{sec:ad-comp} and
\ref{sec:rings-DOs}. Consider a commutative ring $A$
containing $\Q$. There are no finiteness assumptions on $A$. Let
$B := A[[t_1, \ldots, t_n]]$, the ring of power series over $A$ in
$n$ variables. The ring $\DD_{B / A}$ of $A$-linear differential operators of
$B$ is generated by $B$ and the partial derivatives
$\pa_i = \pa / \pa_{t_i}$.
Given a left  $\DD_{B / A}$-module $M$, its horizontal submodule is the
$A$-module
\[ M^{\mrm{hor}} = \{ \lsp m \in M \mid \pa_i(m) = 0
\ \, \tup{for all} \ \, i \lsp \} \sub M . \]

If $L$ is an $A$-module, then the $B$-module
$M := B \msp \wh{\ot}_A \msp L$ has an
induced left $\DD_{B / A}$-module structure.

\begin{thm} \label{thm:215}
Let $A$ be a commutative ring containing $\Q$,
let $B := A[[t_1, \ldots, t_n]]$, the ring of power series over $A$ in
$n$ variables, and let $\b \sub B$ be the ideal generated by the variables.
Let $M$ be a left $\DD_{B / A}$-module, which is $\b$-adically complete
as a $B$-module. Then the canonical $\DD_{B / A}$-module homomorphism
\[ \th :  B \ \wh{\ot}_A \ M^{\mrm{hor}} \to M   \]
is bijective.
\end{thm}

\begin{proof}
As done in \cite{Be} and \cite{Bj}, our proof is by induction on $n$. For
$n = 1$ this is Theorem \ref{thm:240}.

Now assume that $n \geq 2$, and that the theorem is true for the ring
$B' := A[[t_1, \ldots, t_{n - 1}]]$. Writing $t := t_n$, we have an $A$-ring
isomorphism $B \cong B'[[t]]$. Also, writing
$\pa := \pa_n \in \DD_{B / A}$, the operator $\pa$ commutes with the operators
$t_i$ and $\pa_i$, for all $1 \leq i \leq n - 1$, inside the noncommutative
ring $\DD_{B / A}$.
The ring $\DD_{B / B'}$ is generated by $B$ and $\pa$.

For a subsequence $(i_1, \ldots, i_k)$ of $(1, \ldots, n)$ let
\[ M^{\pa_{i_1}, \ldots, \pa_{i_k}} :=
\{ m \in M \mid \pa_{i_1}(m) = \cdots = \pa_{i_k}(m) = 0 \} \sub M . \]
Let us define the $B'$-module  $M' := M^{\pa} = M^{\pa_n} \sub M$.
The commutation relations mentioned above imply that $M'$ is actually a
$\DD_{B' / A}$-submodule of $M$.
By definition we have
$M^{\mrm{hor}} = M^{\pa_1, \ldots, \pa_{n}}$,
and hence
$M^{\mrm{hor}} = (M')^{\pa_1, \ldots, \pa_{n - 1}}$.

The ideal $(t) \sub B$ is contained in the ideal $\b$, and $M$ is $\b$-adically
complete. Corollary \ref{cor:420}
tells us that $M$ is $(t)$-adically complete.
Applying Theorem \ref{thm:240} to the ring homomorphism
$B' \to B = B'[[t]]$, and the $(t)$-adically complete
$\DD_{B / B'}$-module $M$, we conclude that the
$\DD_{B / B'}$-module homomorphism
$B \msp \wh{\ot}_{B'} \msp M' \to M$
is  bijective.

Let $\b' \sub B'$ be the ideal $(t_1, \ldots, t_{n - 1})$.
Using Corollary  \ref{cor:420} once more, we know that $M$ is $\b'$-adically
complete. Now $M' = \opn{Ker}(\pa)$, so it is a closed submodule of $M$ for the
$\b'$-adic topology. According to Corollary \ref{cor:421}, $M'$ is a
$\b'$-adically complete $B'$-module.
Since Theorem \ref{thm:215} is assumed to be true for $B'$, and
since $(M')^{\pa_1, \ldots, \pa_{n - 1}} = M^{\mrm{hor}}$,
it follows that the
$\DD_{B' / A}$-module homomorphism
$B' \msp \wh{\ot}_{A} \msp M^{\mrm{hor}} \to M'$
is bijective.

The associativity of complete tensor products implies that there are $B$-module
isomorphisms
\[ B \msp \wh{\ot}_{A} \msp M^{\mrm{hor}} \cong
B \msp \wh{\ot}_{B'} \msp (B' \msp \wh{\ot}_{A} \msp M^{\mrm{hor}})
\cong B \msp \wh{\ot}_{B'} \msp M' \cong M . \]
The composed isomorphism is precisely
$\th : B \msp \wh{\ot}_{A} \msp M^{\mrm{hor}} \to M$.
\end{proof}

In the corollaries below we assume the setting of the theorem.

Define the $A$-module
$\widebar{M} := M / \b \cd M$.
Consider the inclusion $\ep : M^{\mrm{hor}} \inj M$ and the projection
$\pi : M \surj \widebar{M}$. Let
$\ga : M^{\mrm{hor}} \to \widebar{M}$
be $\ga := \pi \circ \ep$.

\begin{cor} \label{cor:295}
The $A$-linear homomorphism
$\ga : M^{\mrm{hor}} \to \widebar{M}$ is bijective.
\end{cor}

\begin{proof}
This is obvious for the induced module
$N := B \msp \wh{\ot}_A \msp M^{\mrm{hor}}$.
According to Theorem \ref{thm:240} there is an isomorphism of left
$\DD_{B / A}$-modules $N \cong M$, so the same holds for $M$.
\end{proof}

In view of Corollary \ref{cor:295}, we can define the $A$-linear bijection
$\si : \widebar{M} \iso M^{\mrm{hor}}$ to be the inverse of $\ga$.
We then have:

\begin{cor} \label{cor:296}
Given an element $\widebar{m} \in \widebar{M}$, the element
$m := \si(\widebar{m}) \in M$ is the unique solution in $M$ of the differential
equation
$\pa_1(m) = \cdots = \pa_n(m) = 0$ with initial condition
$\pi(m) = \widebar{m}$.
\end{cor}

\begin{cor} \label{cor:462}
There is an $A$-linear direct sum decomposition
$M = M^{\mrm{hor}} \oplus \b \cd M$.
\end{cor}

\begin{proof}
The same as the proof of Corollary \ref{cor:460}(1).
\end{proof}

A $B$-module $M$ is called {\em $\b$-adically free} if $M$ is the $\b$-adic
completion of a free $B$-module.
It is known that a $\b$-adically free $B$-module $M$ is isomorphic to the
module of $\b$-adically decaying functions
$\opn{F}_{\mrm{dec}}(Z, B)$ for some set $Z$; see \cite[Section 2]{Ye3}.

\begin{cor} \label{cor:452}
If $M / \b \cd M$ is a free $A$-module,
then $M$ is a $\b$-adically free $B$-module.
\end{cor}

\begin{proof}
The $A$-module $M^{\mrm{hor}} \cong M / \b \cd M$, so it is free. Then the
$B$-module $B \ot_A M^{\mrm{hor}}$ is free, and its $\b$-adic completion
$B \msp \wh{\ot}_{A} \msp M^{\mrm{hor}}$
is $\b$-adically free.
\end{proof}

The following corollary is
\cite[Proposition 8.9]{Ka}, which is phrased in terms of connections.

\begin{cor} \label{cor:451}
Assume $A$ is a field and $M$ is a finitely generated $B$-module.
Then $M \cong B^{\oplus r}$ as $\DD_{B / A}$-modules, for some natural number
$r$.
\end{cor}

\begin{proof}
Here $M / \b \cd M \cong A^{\oplus r}$ for some $r$, and
$B^{\oplus r}$ is already complete.
\end{proof}

\begin{rem} \label{rem:305}
Suppose $\K$ is a field of characteristic $0$, $X$ is a smooth $\K$-scheme,
and $\MM$ is a sheaf of left $\DD_{X / \K}$-modules, which is coherent as an
$\OO_X$-module. According to \cite[Proposition 1.a]{Be} or
\cite[Theorem 1.4.10]{HTT}, $\MM$ is a locally free $\OO_X$-module. The direct
proof in \cite{HTT} is very easy. They consider a closed point $x \in X$, the
local ring $A :=  \OO_{X, x}$, and the $A$-module $M := \MM_x$. It suffices to
prove that $M$ is a free $A$-module. This is done by a calculation
almost identical to the the one in the proof of Theorem \ref{thm:240}.

It is possible to deduce the algebraic result above from the corresponding
complete result, namely Theorem \ref{thm:215}. Let $\what{A}$ be the completion
of the local ring $A$ at its maximal ideal, and let $\what{M}$ be the
completion of the module $M$. Since $A \to \what{A}$ is faithfully flat, it
suffices to prove that $\what{M}$ is a free $\what{A}$-module. The completed
module $\what{M}$ has a $\DD_{\what{A} / \K}$ structure.
The residue field $\K'$ of $A$ is a finite separable extension of $\K$, so it
can be lifted uniquely into $\what{A}$, and then
$\what{A} \cong \K'[[t_1, \ldots, t_n]]$.
Now Theorem \ref{thm:215} can be applied.

The reverse direction -- namely using \cite[Theorem 1.4.10]{HTT} to prove
Theorem \ref{thm:215} -- does not work, because we do not know a priori that
$\what{M}$ descends to some $A$-module $M$.

Furthermore, there are plenty of nonisomorphic $\DD_{A / \K}$-modules
that trivialize when passing to the completion $\what{A}$.
To see this explicitly, let's assume that $\K = \C$, $n = 1$ and
the rank of $M$ is $1$.
For every $\la \in \C$ there is a $\DD_{A / \C}$-module
$M_{\la}$, which is a free $A$-module of rank $1$ with basis $e_{\la}$, and
such that $\pa(e_{\la}) = \la \cd e_{\la}$.
When passing to the completion, the $\what{A}$-module
$\what{M}_{\la}$ is still free with basis $e_{\la}$, but now the ring
$\what{A}$ contains the invertible element
$u_{\la} := \opn{exp}(- \la \cd t)$, and the element
$u_{\la} \cd e_{\la} \in \what{M}_{\la}$ is a horizontal basis,
so that $\what{M}_{\la}$ is a trivial $\DD_{\wh{A} / \C}$-module.
\end{rem}

\appendix
\section{An Error in a Book by Bj\"{o}rk}
\label{appendix}

While studying the literature on $\DD$-modules, we looked at the proof of
Theorem 1.1.25 in the book \cite{Bj} by J.-E. Bj\"{o}rk. This theorem is an
analytic version of Corollary \ref{cor:414}, which is a special case of Theorem
\ref{thm:411}. While reading the proof, we discovered an error in it.
In this appendix we shall explain the error, and how it can be resolved.

Before doing so, we should address a question of etiquette. As customary, we
wanted to communicate our discovery to the author of the book, giving him the
possibility to address it. Unfortunately, Bj\"{o}rk passed away several
years ago. Therefore we communicated his student Rolf K\"allstr\"om, who is
working in this same area. He confirmed that there is indeed an
error, that this error had not been observed previously, and that we ought
to write about it.

Now to the mathematics. We prefer to use our notation, and to simplify the
situation as much as possible, so that the salient points will be clear. We
also want to treat two settings at once:
\begin{itemize}
\item[($*$)] The setting of our paper, specialized: $A = \K$ is a field
of characteristic $0$, and $B = \K[[t]]$. Convergence is for the
$\b$-adic topology, where $\b = (t) \sub B$.

\item[($**$)] The complex analytic setting that is treated in the book
\cite{Bj}. Here $A = \C$, the field of complex numbers with its standard norm,
and $B$ is the subring of $\C[[t]]$ consisting of convergent power series,
namely series $b = \sum_{k \geq 0} \la_k \cd t^k$ with coefficients
$\la_k \in \C$ satisfying the asymptotic condition
$\opn{limsup}_{k \to  \infty} (\abs{\la_k})^{1 / k}  < \infty$.
In other words, $B$ is the ring of germs of analytic functions in one variable
at the origin. It is a noetherian local ring, with maximal ideal $\b$
generated by $t$.
\end{itemize}
The ring of differential operators $\DD_{B / A}$ has the same structure in both
settings. As before we write $\pa := \smfrac{\partial}{\partial t}$.

Suppose $M$ is a left $\DD_{B / A}$-module that's finitely generated as a
$B$-module. We can assume that $M$ is a free $B$-module of rank $r$, with basis
$\bs{m}' = (m'_1, \ldots, m'_r)$; for this we can either rely on
\cite[Theorem  1.4.10]{HTT}, or on the first half of the proof of
\cite[Theorem 1.1.25]{Bj}, which is correct (and is similar to the proof in
\cite{HTT}). The issue is how to replace the basis $\bs{m}'$ with another basis
$\bs{m}$, which is horizontal, and is congruent to $\bs{m}'$ modulo the maximal
ideal $\b \sub B$.
In other words, we want to find a matrix
$\bs{g} = \bs{g}(t) \in \opn{GL}_{r}(B)$,
$\bs{g} \equiv \bs{1} \msp \opn{mod} \lsp t$,
such that
\begin{equation} \label{eqn:445}
\pa(\bs{g} \cd \bs{m}') = 0 .
\end{equation}
Then the basis $\bs{m} := \bs{g} \cd \bs{m}'$ of $M$ will be horizontal.

The difficulty only arises when $r \geq 2$, so let us assume this.
Bj\"{o}rk considers the matrix
$\bs{b} \in \opn{Mat}_{r}(B)$
such that $\pa(\bs{m}') = \bs{b} \cd \bs{m}'$.
The matrix $\bs{b}$ has a power series expansion
$\bs{b} = \sum_{k \geq 0} \msp t^k \cd \bs{a}_k$
with $\bs{a}_k \in \opn{Mat}_{r}(A)$.
Convergence of this series in the setting ($*$) is easy; we did not verify
convergence in the setting ($**$), but presumably this is automatic for
standard
reasons.

Next Bj\"{o}rk defines the matrix
\[ \bs{c} := \sum_{k \geq 0} \msp (k+1)^{-1} \cd t^{k + 1} \cd \bs{a}_k
\in \opn{Mat}_{r}(B) , \]
which satisfies $\pa(\bs{c}) = \bs{b}$.
Again, convergence is obvious in the adic setting ($*$), and
requires justification in the analytic setting ($**$).
Then he defines the matrix
\[ \bs{g} := \opn{exp}(- \bs{c}) =
\sum_{k \geq 0} \msp (k !)^{-1} \cd (- \bs{c})^k \in \opn{Mat}_{r}(B) . \]
Observe that this sum converges in the adic setting ($*$) because the
matrix $\bs{c}$ is divisible by $t$. In the complex analytic setting
($**$) convergence is because of the rapid decay of $(k !)^{-1}$.

The erroneous claim of Bj\"{o}rk (implicit in the text, yet present
in formula (2) there) is that the matrix
$\bs{g} \in \opn{Mat}_{r}(B)$ satisfies the differential equation
\begin{equation} \label{eqn:300}
\pa(\bs{g}) =^{\tup{false}} - \bs{g} \cd \bs{b} ,
\end{equation}
with initial condition
$\bs{g}(0) = \bs{1} \in \opn{Mat}_{r}(A)$. {\em For the differential equation
\tup{(\ref{eqn:300})} to hold it is sufficient, and very likely also necessary,
that the matrices $\bs{a}_k$ commute with each other}.

Anyhow, here is a counterexample: take the ring $A := \Q$, so $B = \Q[[t]]$, and
the rank $r := 2$. The matrix
$\bs{c} = \bs{c}(t) \in \opn{Mat}_{2}(B)$ is
$\bs{c} := t \cd \bs{a}_1 + t^2 \cd \bs{a}_2$, with
$\bs{a}_1 := \sbmat{0 & 1 \\ 0 & 0}$ and
$\bs{a}_2 := \sbmat{1 & 0 \\ 0 & 0}$.
Recall that $\bs{b} = \pa(\bs{c})$ and
$\bs{g} = \opn{exp}(- \bs{c})$.
A calculation shows that
$\pa(\bs{g}) \neq - \bs{g} \cd \bs{b}$.

Bj\"{o}rk proceeds as follows. He takes the sequence
$\bs{m} := \bs{g} \cd \bs{m}'$ of elements of $M$. Here is our
interpretation of Bj\"{o}rk's formula (2):
\begin{equation} \label{eqn:301}
\pa(\bs{m}) = \pa(\bs{g} \cd \bs{m}') =
\pa(\bs{g}) \cd \bs{m}' + \bs{g} \cd \pa(\bs{m}') =^{\tup{false}}
- \bs{g} \cd \bs{b} \cd \bs{m}' + \bs{g} \cd \bs{b} \cd \bs{m}' = 0 .
\end{equation}
The false equality $=^{\tup{false}}$ relies on the false formula
(\ref{eqn:300}).

In the setting ($**$) of complex analytic functions, the correct way
to solve the ODE (\ref{eqn:445}) is by appealing to the standard existence
results. Presumably this is what Deligne had in mind when he wrote
"Il est bien connu que..." regarding \cite[Theorem 2.17]{De}.
In \cite[section 4.2]{HTT}, when discussing analytic $\DD$-modules,
they refer to the Frobenius Theorem.

In our $t$-adic setting ($*$), Lemma \ref{lem:275} and its proof
give us an explicit formula for the matrix $\bs{g}$ solving the ODE
(\ref{eqn:445}). Given the basis
$\bs{m}' = (m'_1, \ldots, m'_r)$ of $M$, let
$m_i :=  \psi(m'_i) \in M$ and
$\bs{m} := (m_1, \ldots, m_r)$,
which is a horizontal basis of $M$. Then define
$\bs{g} \in \opn{GL}_{r}(B)$ to be the matrix such that
$\bs{g} \cd \bs{m}' = \bs{m}$. This matrix $\bs{g}$ will solve the ODE
(\ref{eqn:445}).

An alternative approach for obtaining the matrix $\bs{g}$ is by {\em
$1$-dimensional nonabelian multiplicative integration}, also known as {\em path
ordered exponential}.
This method would present $\bs{g}$ as the limit of {\em Riemann products}, cf.\
\cite[Chapter 3]{Ye5}.


\end{document}